\pdfoutput=1
\documentclass{amsproc}

\newtheorem{theorem}{Theorem}[section]

\theoremstyle{definition}

\newtheorem{example}[theorem]{Example}

\theoremstyle{remark}
\newtheorem{remark}[theorem]{Remark}

\numberwithin{equation}{section}

\begin{document}

\title{Einstein warped-product manifolds and the screened Poisson equation}

\author[A. Pigazzini]{Alexander Pigazzini}
\address{Mathematical and Physical Science Foundation, Sidevej 5, 4200 Slagelse, Denmark}
\email{pigazzini@topositus.com}

\author[L. Lussardi]{Luca Lussardi}
\address{Dipartimento di Scienze Matematiche ``G.L.\,Lagrange'', Politecnico di Torino, c.so Duca degli Abruzzi 24, I-10129 Torino, Italy}
\email{luca.lussardi@polito.it}

\author[M. Toda]{Magdalena Toda}
\address{Department of Mathematics and Statistics, 1108 Memorial Circle, MS 1042, Texas Tech University, Lubbock TX 79409-1042, U.S.A}
\email{magda.toda@ttu.edu}

\author[A. D{e}{B}enedictis]{Andrew D{e}{B}enedictis}
\address{The Pacific Institute for the Mathematical Sciences \\  and \\ Faculty of Science, Simon Fraser University,  Burnaby, British Columbia, V5A 1S6, Canada}
\email{adebened@sfu.ca}

\subjclass[2020]{Primary 53C25, 53C21}
\date{May 3, 2024.}


\keywords{Warped-product manifolds; constant negative curvature; screened Poisson equation; time-independent Klein-Gordon equation.}

\begin{abstract}
We study a particular type of Einstein warped-product manifold where the warping function must satisfy the homogeneous version of the screened Poisson equation. 
Under these assumptions, we show that the dimension of the manifold, the (constant negative) Ricci curvature and the screened parameter are related through a quadratic equation.
\end{abstract}

\maketitle

\section{Introduction and Preliminaries}

\subsection{Einstein warped product manifolds}

Warped-product manifolds, first introduced by Bishop and O'Neill (\cite{Bishop}), have gained significant importance in the field of differential geometry and are widely researched in the domain of General Relativity, particularly in relation to generalized Friedmann-Robertson-Walker spacetimes. A wealth of properties for warped product manifolds and submanifolds have been presented in various studies (\cite{Chen}, \cite{Kim} and \cite{KK}).

To construct a warped-product manifold, consider two semi-Riemannian manifolds $(B, g_B)$ and $(F, g_F)$, and let $\tau$ and $\sigma$ be the projections of $B \times F$ onto $B$ and $F$ respectively. The warped-product $M=B \times_f F$ is the manifold $B \times F$ with the metric tensor $g=\tau^*g_B+f^2 \sigma^*g_F$, where $^*$ denotes the pullback and $f$ is a positive smooth scalar function on $B$, known as the warping function.

In more explicit terms, if $X$ is tangent to $B \times F$ at $(a,b)$ (where $a$ is a point on $B$ and $b$ is a point on $F$), then:
\\
{\centerline {$\langle X,X \rangle=\langle d\tau(X),d\tau(X)\rangle+f^2(a)(d\sigma(X),d\sigma(X))$.}}
\\
\\
In this context, $B$ is referred to as the \textit{base-manifold} of $M=B \times_f F$ and $F$ is the \textit{fiber-manifold}. If $f=1$, then $B \times_f F$ simplifies to a semi-Riemannian product manifold. The leaves $B \times b = \sigma^{-1}(b)$ and the fibers $a \times F =\tau^{-1}(a)$ are Riemannian submanifolds of $M$. Vectors tangent to leaves are termed horizontal and those tangent to fibers are termed vertical. The orthogonal projection of $T_{(a,b)}M$ onto its horizontal subspace $T_{(a,b)}(B \times b)$ is denoted by $\mathcal{H}$, and $\mathcal{V}$ denotes the projection onto the vertical subspace $T_{(a,b)}(a \times F)$ (\cite{O'Neill}).

If $M$ is an $n$-dimensional manifold with metric tensor $g_M$, the Einstein condition gives $Ric_M = \lambda g_M$ for some real constant $\lambda$, where $Ric_M$ represents the Ricci tensor of $g_M$. An Einstein manifold with $\lambda = 0$ is considered Ricci-flat.

With this understanding, we find that a warped-product manifold $(M, g_M)=(B,g_B)\times_f(F,g_F)$ (where ($B, g_B$) is the base-manifold, ($F, g_F$) is the fiber-manifold), with 
$g_M=g_B+f^2 g_F$,  is Einstein if and only if (refer to \cite{Chen}):
\\
\begin{equation}
Ric_M=\lambda g_M \Longleftrightarrow\begin{cases} 
 Ric_B- \frac{m}{f}Hess(f)= \lambda g_B  \\  Ric_F=\mu g_F \\ f \Delta f+(m-1) |\nabla f|^2 + \lambda f^2 =\mu,
\end{cases} \label{eq:cases}
\end{equation}
\\
where $\lambda$ and $\mu$ are constants, $m$ is the dimension of $F$, $Hess(f)$, $\Delta f$ and $\nabla f$ are, 
respectively, the Hessian, the Laplacian (given by $tr(Hess(f))$) and the gradient of $f$ with respect to  $g_B$, with $f:B \rightarrow \mathbb{R}^+$ a smooth positive function. Regarding the constant $\mu$, for the system to produce a manifold that obeys the Einstein condition $Ric_{M}=\lambda g_{M}$ the last expression in (\ref{eq:cases}) must be constant. The constant $\mu$ is specifically associated with the Ricci curvature of the fiber manifold $F$. The third equation in the above system (which includes $\mu$) ensures the consistency of the curvature conditions over the entire manifold $M$, relating the warping function $f$, its derivatives and the curvatures $\lambda$ and $\mu$.
\\
\\
Contracting the first equation of (1.1) we get: 
\\
\begin{equation}
R_Bf- m\Delta_{g_B} f=n f \lambda,
\end{equation}
where $n$ and $R_B$ are the dimension and the scalar curvature of $B$, respectively, then (1.1), contracting also the second equation, becomes:
\\
\begin{equation}
R_M=\lambda (n+m) \Longleftrightarrow\begin{cases} 
R_Bf- m\Delta_{g_B} f=n f \lambda \\ R_F=\mu m \\ f \Delta_{g_B} f+(m-1) |\nabla_{g_B} f|^2 + \lambda f^2 =\mu.
\end{cases}
\end{equation}

\begin{remark} It is a well-established fact that an Einstein warped-product manifold, equipped with a Riemannian metric and a Ricci-flat fiber-manifold ($F$), cannot admit positive Ricci curvature solutions (see \cite{BL}). It can only accommodate zero or negative solutions. This conclusion is derived from the Bonnet-Myers theorem. Essentially, this theorem states that if $Ric_M>0$, then $M$ is compact. However, if $V$ is a vertical vector (that is, a tangent vector on the fiber-manifold), then the formula for the Ricci curvature of warped product manifolds (as outlined in \cite{Besse}) yields: 
\\
\centerline{$Ric_M(V, V)=Ric_F(V, V)-|V|^2(\frac{\Delta f}{f}+(m-1)\frac{|\nabla f|^2}{f^2})$.}
\\
If the fiber-manifold $F$ is Ricci-flat, then it follows that at the point where the warping function $f$ reaches its minimum on the base-manifold (which is guaranteed to exist due to compactness), the condition $Ric_M(V, V) \le 0$ is satisfied.
\\
In this paper we are only interested in cases $R_F=0$ with $Ric_M<0$.
\end{remark}

\subsection{Screened Poisson equation and Klein-Gordon equation}

In the realm of mathematics, the screened Poisson equation is a specific type of partial differential equation, conventionally represented as $[\Delta-\lambda^2]\phi(r)=-\psi(r)$. Here, $\Delta$ denotes the Laplace-Beltrami operator, $\lambda$ is a constant, $\psi$ is a function of position ($r$) known as the ``source function'', and $\phi$ is the function to be determined.
When $\lambda$ is zero, this equation simplifies to the Poisson equation.

A class of solutions to the inhomogeneous screened Poisson equation for a general $\psi$ can be found using the method of Green functions. In this context, the Green function $G$ is defined by the equation $[\Delta-\lambda^2]G(r)=-\delta^3(r)$. This equation frequently appears in various fields of physics, including the Yukawa theory of mesons, where the Green function represents the Yukawa potential \cite{ref:ypot}. The screened Poisson equation also appears in plasma screening, for example in limits of the Debye-H\"{u}ckel theory \cite{ref:dhthy} or the Thomas-Fermi theory \cite{ref:tfthy}. The equation also has applicability in granular fluid flow \cite{ref:fflow1}, \cite{ref:fflow2} and also emerges in the Klein–Gordon equation, which, aside from its importance in physics, holds a significant role in the theory of integrable systems or geometric PDE. This is a rapidly growing subfield at the intersection of differential geometry, differential equations, and mathematical physics.

The Klein-Gordon equation can be transformed into the form of a Schrödinger equation (as two coupled differential equations, each of first order in time) \cite{ref:schro}. However, the Klein-Gordon equation carries a distinct meaning. Typically, in a physics course setting, it is presented with separated time and space. For instance, with respect to natural units and the Minkowski metric (+, -, -, -) (up to a potential change in sign), and in the case of time-independence ($[\Delta-\frac{m^2c^2}{\hbar}]\phi(r)=0$), the Klein-Gordon equation is equivalent to the homogeneous form ($\psi=0$) of the screened Poisson equation: $[\Delta-\lambda^2]\phi(r)=0$. Here, the constant $\lambda^2=\frac{m^2c^2}{\hbar}$, where $c$ and $\hbar$ represent the speed of light and Planck's constant, respectively.

\section{Einstein warped-product manifolds and homogeneous form of the screened Poisson equation} 
\begin{theorem} \label{thm:1} Let $M$ be an Einstein warped-product manifold with constant negative Ricci curvature ($\lambda <0$) and $m$-dimensional Ricci-flat fiber manifold ($F$), and let the base-manifold ($B$) be a $2$-dimensional manifold with negative constant Gaussian curvature $K$. There exists a
warping function $f$ satisfying the homogeneous form of the Screened Poisson equation, i.e. $[\Delta_{g_B}-1]f=0$, if and only if
$dim(M) = 2 + m$ and the constant $\lambda$ of curvature are related by the following quadratic equation:
\begin{equation}
\lambda^2(2-m)+\lambda(1+\frac{3m}{2}-\frac{m^2}{2})+\frac{m^2}{2}+\frac{m}{2}=0. 
\end{equation}
\end{theorem}

\begin{proof} Let consider (1.3) for a $2$-dimensional base-manifold, i.e. $n=2$:
\\
\begin{equation}
R_M=\lambda (2+m) \Longleftrightarrow\begin{cases} 
R_Bf- m\Delta_{g_B} f=2 f \lambda \\ R_F=0 \\ f \Delta_{g_B} f+(m-1) |\nabla_{g_B} f|^2 + \lambda f^2 =0.
\end{cases}
\end{equation}
\\
Considering initially, for ease of calculation, that $B$ is a surface with constant Gaussian curvature $K = -1$, it is simple to notice that a general approach to this problem can be traced back to the following system of equations:
\\
\begin{equation}
\begin{cases}
\Delta_{g_B} f=p(f)q(f) \\ |\nabla_{g_B} f|^2 = p(f)^2.
\end{cases}
\end{equation}
\\
Consider two real-valued functions, $p > 0$ and $q$, defined over an interval $I\subset\mathbb{R}$. Our objective is to ascertain the existence of a solution, $f$, for the system
$
|\nabla f|^2 = p(f)^2, \Delta f = p(f)q(f)
$
on a certain (nonempty) open subset of the Poincaré upper half-plane, which is a Riemannian surface exhibiting a Gauss curvature $K$ that remains consistently equal to $-1$, commonly referred to as a pseudospherical surface. We assert that if and only if the functions $p$ and $q$ satisfy the following differential equation, then a solution $f$ exists:
\\
\begin{equation}
p(t)p''(t)-p'(t)^2 + 2q(t)p'(t)-p(t)q'(t)-q(t)^2 + 1 = 0.
\end{equation}
\\
Working with Cartan differential forms associated to a Cartan moving frame is the most natural setting here. Let's start by noting that if such $f$ exists, then the metric $g$ can be written in the form $g = {\omega_1}^2 + {\omega_2}^2$ where $\omega_1 = (\mathrm{d}f)/p(f)$ and $\ast\mathrm{d}f = p(f)\,\omega_2$. Here $\ast$ is the usual Hodge operator.  
\\
Since $\mathrm{d}(\ast\mathrm{d}f) = \Delta f\,\omega_1\wedge\omega_2$, it follows that $\mathrm{d}(p(f)\,\omega_2) = p(f)q(f)\,\omega_1\wedge\omega_2$. This relates to the equality $\mathrm{d}(p(f)) = p'(f)\,\mathrm{d}f=p(f)p'(f)\,\omega_1$, and thus we must obtain that $\mathrm{d}\omega_2 = \bigl(q(f)-p'(f)\bigr)\,\omega_1\wedge\omega_2$. Further, since $\mathrm{d}\omega_1 = -\omega_{12}\wedge\omega_2$ and $\mathrm{d}\omega_2 =\omega_{12}\wedge\omega_1$, it follows that $\omega_{12} = \bigl(p'(f)-q(f)\bigr)\,\omega_2$. The expansion of the Gauss-Codazzi equation $\mathrm{d}\omega_{12} = K\,\omega_1\wedge\omega_2 = -\omega_1\wedge\omega_2$ allows us to obtain equation (2.4).
\\
Conversely, assuming that $p$ and $q$ satisfy (2.4), we proceed to examine the following equations:
\\
\\
{\centerline{$\omega_1 = \mathrm{d}f/p(f)$, and}}
\begin{equation}
\mathrm{d}\omega_2 = \bigl(q(f)-p'(f)\bigr)\,\omega_1\wedge\omega_2
= \bigl(q(f)-p'(f)\bigr)/p(f)\,\mathrm{d}f\wedge\omega_2.
\end{equation}
Considering a positive function $s$ (according to the linear Ordinary Differential Equation theory), defined on interval $I$ (unique up to a constant multiple) we get:
\\
\begin{equation}
s'(f) = s(f) \bigl(q(f)-p'(f)\bigr)/p(f).
\end{equation}
\\
The above equations (2.5) and (2.6) imply that $\mathrm{d}\bigl(\omega_2/s(f)\bigr)=0$, therefore, if we assume the domain to be simply connected, for a certain function $h$ we have $\omega_2 = s(f)\,\mathrm{d}h$. 
\\
Consequently, equations (2.4) and (2.6) indicate that the metric:
\\
\\
{\centerline{$g = \left(\frac{\mathrm{d}f}{p(f)}\right)^2 + \left(s(f)\,\mathrm{d}h\right)^2$}}
\\
\\
on $I\times\mathbb{R}$ (with coordinates $f$ and $h$), maintains a constant Gaussian curvature equal to $-1$, thereby enabling an isometric immersion into a domain within the Poincaré upper half-plane.
\\
Therefore we showed that, if satisfied, equation (2.4) determines the existence of a solution for $f$ such that $M$ is an Einstein warped-product  manifold with negative constant Ricci curvature, $m$-dimensional Ricci-flat fiber-manifold $F$ and the base-manifold ($B$) a $2$-dimensional hyperbolic manifold (i.e., with Gaussian curvature $K = -1$).
\\
Our aim is to show, in a more general case, (that is, considering $B$ a $2$-dimensional manifold with negative constant Gaussian curvature, not necessarily $K=-1$), that if the warping function $f$, defined on $B$, satisfies $\Delta_ {g_B} f = f$, then between the dimension of $M$ and the constant $\lambda$, there is a second order equation that relates them.
\\
In case the given metric $g_B$ has a constant Gaussian curvature $K<0$, rescaling the metric $\bar g_B = (-K)\, g_B$, we obtain the corresponding Gauss curvature in the new metric, namely $\bar K = -1$. So, from (2.3), we will have:
\\
\begin{equation}
\begin{cases} 
\Delta_{\bar g_B} f = (-K)^{-1}\,\Delta_{g_B} f \\ |\nabla_{\bar g_B} f|^2 = (-K)^{-1}\,|\nabla_{g_B} f|^2 ,
\end{cases}
\end{equation}
and from this, now it is simple to determine $\bar p$ and $\bar q$. Therefore, in the most general scenario, we use (2.4) with $(\bar p,\bar q)$.
\\
That said, considering (1.3) with the settings expressed in \textit{Theorem \ref{thm:1}}, we obtain:
\\
\begin{equation}
\begin{cases} 
 R_B = 2\lambda + m  \\  R_F=0 \\ |\nabla_{g_B} f|^2 =-\frac{f^2(\lambda + 1)}{m-1},
\end{cases}
\end{equation}
and the second equation of (2.3) is equal to:
\\
\begin{equation}
 |\nabla_{g_B} f|^2 =\frac{f^2(-\lambda - 1)}{m-1}=p(f)^2,
\end{equation}
then
\begin{equation}
 \nabla_{g_B} f=\frac{f\sqrt{-\lambda - 1}}{\sqrt{m-1}}=p(f),
\end{equation}
therefore from (2.10) we deduce that $\lambda$ must be $\le -1$.
\\
Now, considering (2.7) (which corresponds to (2.3) with rescaled metric), we obtain:
\\
\begin{equation}
\begin{cases} 
\frac{f}{-K} = \bar p(f) \bar q(f) \Longleftrightarrow \bar q(f)=\frac{\sqrt{m-1}}{\sqrt{-\lambda-1}}\frac{1}{\sqrt{-K}} \\ \frac{f\sqrt{-\lambda - 1}}{\sqrt{m-1}}\frac{1}{\sqrt{-K}} =\bar p(f).
\end{cases}
\end{equation}
\\
If we consider (2.4):
\\
\begin{equation}
\begin{cases} 
\bar p''(f)=0 \\ -\bar p'(f)^2=\frac{\lambda+1}{m-1}(- K)^{-1}\\ 2 \bar q(f) \bar p'(f)=2  (- K)^{-1} \\ \bar q'(f)=0 \\ - \bar q(f)^2=\frac{1-m}{\lambda + 1}\frac{1}{K}.
\end{cases}
\end{equation}
\\
Since $K=\frac{R_B}{2}$, from the first equation of (2.8), we get $-K=-(\lambda +\frac{m}{2})$, hence (2.4) is equivalent to:
\\
\begin{equation}
\frac{\lambda+1}{m-1}+2+\frac{m-1}{\lambda+1}- \Bigl(\lambda +\frac{m}{2}\Bigr)=0,
\end{equation}
or
\\
\begin{equation}
\lambda^2(2-m)+\lambda \Bigl(1+\frac{3m}{2}-\frac{m^2}{2}\Bigr)+\frac{m^2}{2}+\frac{m}{2}=0.
\end{equation}\end{proof}
The second degree equation (2.14) relates the dimension $m$ of the fiber-manifold ($F$) and the constant $\lambda$ (where $\lambda \in \mathbb{R}^-$), in order to define an Einstein warped-product manifold ($M$) with constant negative Ricci curvature and a $2$-dimensional base-manifold ($B$) with constant negative Gaussian curvature, on which the defined warping function ($f$), satisfies the homogeneous form of Screened Poisson equation.

\begin{remark}
More generally we can consider $[\Delta_{g_B}-\beta]f=0$, with constant $\beta \in \mathbb{R}^+$, so (2.14) will be:
\\
\begin{equation}
\lambda^2(2-m)+\lambda \Bigl( \beta+\frac{3m\beta}{2}-\frac{m^2\beta}{2}\Bigr)+m^2 \Bigl(1-\frac{\beta^2}{2}\Bigr)+m \Bigl(\frac{5\beta^2}{2}-2 \Bigr)=0. 
\end{equation}
\\
We have obtained that for this special family of Einstein warped-product manifolds, the second order equation (2.15) establishes a relation, in general, between the dimension (i.e., $2 +m$), the constant of curvature ($\lambda$) and the ``screened'' parameter ($\beta$). These manifolds exist if and only if the equation (2.15) admits at least one negative real value as a solution for the dimension and screened parameter chosen.
\end{remark}

\begin{example}
If we consider that $B$ is, for example, the upper Poincaré half-plane, then the line element of $M$, $ds^2$, will be:
\\
\begin{align}
ds^2 & =  g_{B_{ij}}dx^{i}dx^{j} +f^2  g_{F_{uv}}dx^u dx^v\,\label{eq:linel} \\[0.1cm]
&=4\frac{dx_1^2+dx_2^2}{(1-x_1^2-x_2^2)^2}+f^2  g_{F_{uv}}dx^{u}dx^{v}\,, \nonumber
\end{align}
where  $u, v=3, 4, 5,...$, and $ g_{F_{uv}}$ represents the components of the (semi)Riemannian Ricci-flat metric on $F$. Then $|g_B|=\frac{16}{(1-x_1^2-x_2^2)^4}$ and $\sqrt{|g_B|}=\frac{4}{(1-x_1^2-x_2^2)^2}$.
\\
From $[\Delta_{g_B}-\beta]f=0 \longrightarrow \nabla_i \partial^i f -\beta f=0$, where $\nabla_i \partial^i f=\frac{1}{\sqrt{|g_B|}}\partial_i(\sqrt{|g_B|}g_{B}^{ij} \partial_j f)$, we get: $\nabla_i \partial^i f=\frac{(1-x_1^2-x_2^2)^2}{4} \partial_i \Bigl[\frac{4}{(1-x_1^2-x_2^2)^2} g^{ij} \partial_j f \Bigr]$.
\\
Now, after some straight-forward computation, we obtain: $\Delta_{g_B} f=\frac{(1-x_1^2-x_2^2)^2}{4}\Bigl [\frac{\partial^2 f}{\partial x^2_1}+\frac{\partial^2 f}{\partial x^2_2}\Bigr ]$.
\\
\\
In summary, if and only if the equation (2.15) will be satisfied, will there exist a function $f$ such that it satisfies the following:
\\
\begin{equation}
\frac{(1-x_1^2-x_2^2)^2}{4}\Bigl[ \frac{\partial^2 f}{\partial x^2_1}+\frac{\partial^2 f}{\partial x^2_2} \Bigr] - \beta f=0.
\end{equation}
\end{example}

\section{Open questions and future directions}

An interesting future direction would be to extend these results to other cases of Einstein-warped product manifolds, wherever there is a well-understood corresponding family of integrable systems (geometric PDEs).

\section*{\small{Acknowledgements}}
The authors would like to thank the anonymous referee for suggestions which have improved the manuscript. We also would like to thank the journal editors.

\bibliographystyle{amsplain}

\begin{thebibliography}{A}

%
%

\bibitem{Bishop} Bishop, R. L., O'Neil, B.: \textit{Manifolds of negative curvature}, \textbf{}Trans. Am. Math. Soc. 145, 1-49 (1969). 

\bibitem{Chen} Chen, B.-Y.: \textit{Differential Geometry of Warped Product Manifolds and Submanifolds}, \textbf{} World Scientific, (2017).

\bibitem{Kim}  Kim, S.: \textit{Warped products and Einstein metrics}, \textbf{}J. Phys. A Math. Gen. 39, L329–333 (2006).

\bibitem{KK}  Kim, D.-S., Kim, Y.H.: \textit{Compact Einstein warped product spaces with nonpositive scalar curvature}, \textbf{}Proc. Am. Math. Soc. 131, 2573–2576 (2003).

\bibitem{O'Neill} O'Neill, B.: \textit{Semi-Riemannian Geometry with applications to Relativity}, \textbf{}Academic Press, (1983).

\bibitem{BL} Leandro, B., Lemes de Sousa, M., Pina, R.: \textit{On the structure of Einstein warped product semi-Riemannian manifolds}, \textbf{} Journal of Integrable Systems, 3(1), (2018).

\bibitem{Besse}  Besse, A. L.: \textit{Einstein Manifolds}, \textbf{}Springer, Berlin (1987).

\bibitem{ref:ypot}  Machleidt, R.: \textit{Phenomenology and Meson Theory of Nuclear Forces}. In: I. Tanihata, H. Toki, T. Kajino,(eds) \textit{Handbook of Nuclear Physics}. Springer, Singapore, (2023).

\bibitem{ref:dhthy} Debye, P., and H\"{u}ckel, E.: \textit{Zur Theorie der Elektrolyte. I. Gefrierpunktserniedrigung und verwandte Erscheinungen}, \textbf{} Physik. Zeit. 24, 185-206, (1923).

\bibitem{ref:tfthy} Ashcroft, N.W. and Mermin, N.D.: \textit{Solid State Physics}, Thomson Learning, Toronto, (1976).

\bibitem{ref:fflow1} Kamrin, K. and G. Koval, G.: \textit{Phys. Rev. Let.}, \textbf{} 108, 178301, (2012).

\bibitem{ref:fflow2} Tanios, R., El Mohtar, S., Knio, O., Lakkis, I.: \textit{ arXiv:1911.10944v1}[math.NA], (2019).

\bibitem{ref:schro} Greiner, W.:  \textit{Relativistic quantum mechanics}, third ed., Springer-Verlag, Berlin, (2000).
\\[0.5cm]




\end{thebibliography}

\end{document}